\documentclass[preprint,12pt]{elsarticle}
\usepackage{latexsym,amssymb}
\usepackage{graphicx}
\usepackage{amsthm}
\usepackage{mathrsfs}
\usepackage{amsmath}

\newcommand{\ld}{\mbox{$L^2 (\R) $ }}

\newcommand{\ldg}{\mbox{$L^2 (G) $ }}

\newcommand{\R}{\mbox{${\mathbb{R}}$}}
\newcommand{\N}{\mbox{${\mathbb{N}}$}}
\newcommand{\Z}{\mbox{${\mathbb{Z}}$}}
\newcommand{\C}{\mbox{${\mathbb  C}$}}

\newcommand{\tends}[1]{\mbox{\space \raise-2mm
\hbox{$\textstyle\longrightarrow\atop\scriptstyle {#1}$} \space}}

\newtheorem{theorem_intro}{Theorem}

\newtheorem*{theorem_h}{Theorem}
\newtheorem{thm}{Theorem}[section]
 \newtheorem{coro}[thm]{Corollary}
 \newtheorem{lem}[thm]{Lemma}
 \newtheorem{prop}[thm]{Proposition}
 \theoremstyle{definition}
 \newtheorem{defn}[thm]{Definition}
 \theoremstyle{remark}
 \newtheorem{rem}[thm]{Remark}
 \newtheorem*{ex}{Example}
 \newtheorem{notat}[thm]{Notation}

\begin{document}
\begin{frontmatter}
\title{Linear independence of translates implies linear independence of
affine Parseval frames on LCA groups }
\author{Sandra Saliani}
\address{Dipartimento di Matematica, Informatica ed Economia\\ Universit\`{a} degli Studi
 della Basilicata\\ Viale dell'ateneo lucano 10, 85100 Po\-ten\-za, ITALY}
\ead{sandra.saliani@unibas.it}

\begin{abstract}
Motivated by Bownik and Speegle's result on linear independence of wavelet
Parseval frames, we consider affine systems (analogous to wavelet systems)
 as a result of an action of a  locally compact abelian (LCA)
group on a separable Hilbert space ${\mathcal H}$.
Compared with the wavelet setting, translations are replaced by the action of a countable,
discrete subgroup $\Gamma$ of a second countable, locally compact abelian
group  $G$, acting as  a group of unitary  operators on ${\mathcal H}$; dilations
 are replaced by integer powers of a unitary
operator $\delta$ onto ${\mathcal H}$. We show that, under some compatibility
conditions between $\delta$ and the action of the group $\Gamma$, the linear
independence of the translates of any function in $L^2(G)$ by elements of $\Gamma$
implies the linear
independence of such affine Parseval frames in ${\mathcal H}$.
\end{abstract}


\end{frontmatter}

\section{Introduction}
The study of  doubly invariant subspaces of measurable vector functions defined in
the unit circle and taking values in a separable Hilbert space, as developed by
Helson in \cite{H64}, has been retrieved in the context of shift invariant
subspaces of $L^2(\R^n)$  by
de Boor, DeVore, Ron \cite{BoDeRo}, and Bownik \cite{Bo},
 leading to what is commonly known as Helson's theorem.

A shift invariant subspace of $L^2(\R^n)$ is any closed subspace
$V\subset L^2(\R^n)$ which is closed under integer shifts, i.e. such that
$f\in V$ implies $T_k f=f(\cdot-k)\in V,$ for all $k\in\Z^n.$
 The Fourier transform is defined as $\hat{f}(\xi)=\int_{\mathbb{R}^n}{f(x)\, e^{-ix\cdot \xi}\,dx}$.
\begin{theorem_h}[Helson]
A closed subspace $V\subset L^2(\R^n)$ is shift invariant if and only if
$$V=\{f\in L^2(\R^n),\, (\hat{f}(\xi-2\pi k))_{k\in \mathbb{Z}^n}\in J(\xi),\,
\text{for a.e.}\; \xi\in\mathbb{T}^n\},$$
where $J$ is a measurable range function
$$J:\mathbb{T}^n\rightarrow\{\text{closed subspaces of}\;\,\ell^2(\Z^n)\}.$$

The correspondence between $V$ and $J$ is one-to-one, under the convention that
the range functions are identified if they are equal a.e.. Furthermore, if
$$
V=\overline{\text{span}}\{T_k\varphi,\; k\in \Z^n, \varphi\in \mathscr{A}\},
$$
for an at most countable set $\mathscr{A}\subset L^2(\R^n),$   then,
$$J(\xi)=\overline{\text{span}}\{(\hat{\varphi}(\xi-2\pi k))_{k\in \mathbb{Z}^n}
,\;  \varphi\in \mathscr{A}\},\quad\text{a.e.}\,\xi\in\mathbb{T}^n.$$
\end{theorem_h}

 This result has led to significant progress in the study of linear independence of
 wavelet systems (see  the work by Bownik and Speegle \cite{BoSp}). Given
 $\psi\in L^2(\R^n)$, a wavelet system  in $\R^n$ is
 $$|\det A|^{j/2} \, \psi(A^j x-k),\quad x\in\R^n,\; k\in\Z^n,\; j\in\Z,$$
 where $A$ is an $n\times n$ integer-valued, non-singular, expansive  matrix.

 Helson's theorem  has been generalized to  locally compact abelian (LCA) groups, by
various authors:  Kamyabi Gol and Raisi Tousi  \cite{KaRa},\cite{KaRa10},
Cabrelli and Paternostro \cite{CaPa}, Bownik and Ross \cite{BoRo}.
 In all these works, albeit with several distinctions, a subgroup  $\Gamma$ of a
 LCA group $G$ acts as a group of  translations on $L^2(G).$
It is also worth to mention the work by Barbieri, Hern\'{a}ndez, and
Paternostro \cite{BHP}, on characterization of invariant spaces
in terms of range functions and a suitable  generalized Zak transform, as well as  Iverson's
work \cite{I} in the same direction.

One of the aims of this work is to extend Helson's theorem in case $G$ is a
LCA, second countable, Hausdorff group and $\Gamma $ is
a closed countable subgroup of  $G,$ with compact dual group
of characters $\widehat\Gamma,$ which acts as a group of unitary operators on
a separable Hilbert space ${\mathcal H},$
meaning
that there is a unitary representation $$\pi:\Gamma\rightarrow
{\mathcal{U}}({\mathcal H}).$$
We do not require $\Gamma$ to be co-compact, but
we assume that the measure $\mu$ on $\widehat\Gamma,$
arising from the spectral theorem applied to the representation $\pi,$ is
absolutely continuous (with respect to the Haar measure).

We obtain in Theorem \ref{hel} a characterization of $\pi$-invariant subspaces (i.e. closed
subspaces $V\subset {\mathcal H}$
such that $\pi(\gamma)(V)\subset V$, for all
$\gamma\in\Gamma$) in terms of the
range function
$$J:\widehat{\Gamma}\rightarrow \{\text{closed subspaces of}\; \ell^2(\Gamma)\}.$$

In the case of translations on $L^2(G)$, the same has been achieved  for uniform lattices in $\R^n$ by Bownik, \cite{Bo}, for uniform lattices in LCA group by Cabrelli and Paternostro, \cite{CaPa}, and for co-compact,
but not necessarily discrete, subgroups by Bownik and Ross \cite{BoRo}.
We stress that, in the opposite,  our generalization applies to discrete, but
not necessarily co-compact, subgroups: as an example one can think of $\Gamma=\Z\times \{0\}$
in the additive group $G=\R\times\Z.$

Our result is then applied to  study
affine systems
in ${\mathcal H}$, (analogous to wavelet systems)
\begin{equation}\label{affine0}
\{\delta^j\;\pi(\gamma)\,f,\quad \gamma\in \Gamma,\;
j\in\mathbb{Z}\},
\end{equation}
 where $f\in {\mathcal H}$ is fixed.
 Here the action of  $\Gamma$
as  a group of unitary  operators on
${\mathcal H}$ replaces integer translations;  integer powers of a unitary
operator $\delta$ onto ${\mathcal H}$ replace dilations. We require that $\delta$
verifies a  compatibility condition with the representation $\pi$
described below.

We suppose that there is  a one to one endomorphism
$\alpha:\Gamma\rightarrow\Gamma$ such that the
subgroup $\alpha(\Gamma)$ has finite index in $\Gamma$, and the
dual endomorphism $\alpha^{\ast}$ onto $\widehat{\Gamma}$ defined by
$[\alpha^{\ast}(\chi)](\gamma)=\chi(\alpha(\gamma))$ is ergodic
with respect to Haar measure
($\bigcap_{n\geq
1}{\alpha^n(\Gamma)=\{0\}}$ is enough to assure ergodicity \cite{Ba4}). Moreover we assume
the following compatibility condition holds
$$
\delta^{-1}\pi(\gamma)\delta=\pi(\alpha(\gamma)),
\quad \text{for all}\; \gamma\in\Gamma.
$$

Affine systems  include wavelet  systems for a particular choice of
the group, the representation, and the map $\alpha;$ for example, in
one dimension, one could take $G=\R,$ $\Gamma=\Z,$ ${\mathcal H}=\ld$,
$\pi(k)=T_k,$ $\delta=D_2$ (dilation by $2$), and $\alpha(k)=2k.$

This approach has been used by Baggett and his collaborators in \cite{Ba}, in the
context of the GMRA.  By the spectral theorem of Stone and von Neumann,
in conjunction with the characterization of projection valued measure, any
subrepresentation of $\pi$, arising from an invariant subspace, is realized as a 
direct integral.
We obtain in Lemma \ref{formulaB} that the multiplicity functions associated
with the subrepresentations of $\pi$ on an invariant subspace $V$ and its
``dilation'' $\delta(V)$ verify the same relation obtained by Bownik and
Rzeszotnik in the case
$G=\R^n$,\cite[Corollary 2.5]{BoRz},
for the corresponding dimension functions. It is worth to recall that the dimension
function of an invariant subspace $V$ is defined as
$\dim_V(\xi)=\dim J(\xi)$, where $J$ is the range function provided by the 
generalization of Helson's theorem. 

In the while, in Proposition \ref{dimult}, we  prove a relation  between the multiplicity function associated
with the subrepresentation of $\pi$ on an invariant subspace $V$ and the dimension function $\dim_V$. To our knowledge, this result is not known in the literature.

All these results allows us to prove, by extending the technique used in the work by Bownik and Speegle, \cite{BoSp},
our main result: 
every  affine Parseval frame \eqref{affine0} in ${\mathcal H}$  is linearly
independent provided
the translations of each function in $L^2(G)$ by elements of $\Gamma$ are
linearly independent.

This last assumption is verified by a large class of groups
including the case $G=\R^n$, and $\Gamma=\Z^n$  considered in \cite{BoSp}.
It is verified also by  LCA groups $G$ without non-trivial
  compact subgroups, since then a much stronger property holds: by a standard argument with Fourier transforms \cite[Theorem 1.2]{EdRo79},
 any function in $L^2(G)$ has linearly independent translates  
(see \cite[Corollary 24.35]{HeRo} for a characterization of such LCA groups). 
See also the work
 of Rosenblatt \cite{Ro08} for a discussion on this topic.

 On the opposite, the assumption  is not verified in the
case of finite groups.

We recall that a sequence $(f_n)_{n\in \Z}$  in a Hilbert space ${\mathcal H}$ is {\em linearly independent}
 if every finite subsequence
of $(f_n)_{n\in \Z}$ is linearly independent, i.e.
$$c_n\in \C, \; n\in F \;\text{finite},\quad\sum_{n\in F}{c_n\, f_n}=0\Rightarrow \;c_n=0,\quad \text{for all}\; n\in F.$$

We decided to label the crucial hypothesis, since we shall assume it several times:
\begin{itemize}
\item[{\bf{(A)}}] For any $ 0\neq f\in L^2(G),$ the sequence of translates
$ \{T_{\gamma}f, \,\gamma\in\Gamma\}$
is linearly independent.
\end{itemize}

As it will be shown in Section \ref{main}, our hypotheses on $\Gamma$ and $\mu$
guarantee that linear independence of translates implies  the linear independence
of the sequence $ \{\pi(\gamma)\psi, \,\gamma\in\Gamma\}$, for any $ 0\neq \psi\in {\mathcal H}.$

 Before stating the main results, we need the definition of a (Parseval) frame.
Frames in a separable Hilbert space provide redundant but stable expansions for
elements of the space itself. Frames play key roles in many settings, such as sampling theory,
wavelet analysis, and time-frequency (Gabor) analysis.
We say that a sequence $(f_n)_{n\in \Z}$ in a Hilbert space ${\mathcal H}$ is a frame if there exist constants
$A,B>0$ such that
$$
A\|f\|^2\leq \sum_{n\in \mathbb{Z}}{|<f,f_n>|^2}\leq B \|f\|^2, \quad \text{for all}\; f\in {\mathcal H}.
$$
If $A=B,$ we say that $(f_n)_{n\in \Z}$ is  a tight frame, if $A=B=1,$
a Parseval frame.

The main results of the paper are
\begin{theorem_intro}\label{main1}
Assume hypothesis {\bf{(A)}}.
Suppose $0\neq \psi\in {\mathcal H}$ and its space of negative dilates
$V_0=\overline{\text{span}}\{\delta^j \,\pi(\gamma)\,\psi,\; {j<0},\, \gamma\in\Gamma\}$
 is $\pi$- invariant.

Then
the affine system
$\{\delta^j\, \pi(\gamma)\,\psi, \;{j\in\Z},\, \gamma\in\Gamma\}$
is linearly independent.
\end{theorem_intro}

\begin{theorem_intro}\label{main2}
Assume hypothesis {\bf{(A)}}.
Suppose $0\neq \psi\in {\mathcal H}$. If
the affine system
$\{\delta^j \pi(\gamma)\,\psi, j\in\Z,\, \gamma\in\Gamma\}$
is a Parseval frame, then it is linearly independent.
\end{theorem_intro}

(It is worth to notice that all results apply to the  case
${\mathcal H}=\ldg,$ since
hypotheses on $G$ imply that $\ldg$ is separable).

We tried to separate those results that do not need neither all the machinery of
representation theory of LCA groups, nor the characterization of $\pi$-invariant spaces
 by Helson from those who do. So after the main  hypotheses  in Section
\ref{Hypo}, we state a first result on linear independence in Section \ref{BS}.
In Section \ref{inv} we extend Helson's theorem to $\pi$-invariant spaces,
and in Section \ref{smult} we prove the main properties of the multiplicity
function. The proof of Theorem \ref{main1} and Theorem \ref{main2} are given in
Section \ref{main} together with a non trivial example.

\section{Hypotheses, notations and useful results}\label{Hypo}
In this section we collect all the hypotheses and notations  needed in this paper. We assume that
$G$ is a locally compact  abelian (LCA), second countable, Hausdorff  group 
 and
$\Gamma\subset G$ is a (closed) countable discrete subgroup of  $G$ with (compact) dual group
of characters
$\widehat\Gamma$. We do not require $\Gamma$ to be co-compact. Note that $\widehat{\Gamma}$ is compact and metrizable, hence separable and second countable.

We suppose that there is a one to one endomorphism $\alpha:\Gamma\rightarrow\Gamma$ such that the
subgroup $\alpha(\Gamma)$ has finite index in $\Gamma$, i.e. the quotient group
$$\Gamma/{\alpha(\Gamma)}$$
has a  finite number of elements, say $N>1$.

We consider the dual endomorphism onto $\widehat{\Gamma},$
$\alpha^{\ast}:\widehat{\Gamma}\rightarrow \widehat{\Gamma},$ defined, in any character
$\chi\in\widehat{\Gamma},$ by
$\alpha^{\ast}(\chi)=\chi\circ\alpha.$ Note that $|\ker{\alpha^{\ast}}|=N.$
We assume that $\alpha^{\ast}$ is ergodic with
respect to the normalized Haar measure $\lambda$ on $\widehat{\Gamma}$.
For example, the latter is verified whenever
 $\bigcup_{n\geq 1}{\ker{{\alpha^{\ast}}^n}}$ is dense in $\widehat{\Gamma}$, which
is equivalent to require $\bigcap_{n\geq
1}{\alpha^n(\Gamma)=\{0\}},$ \cite{Ba4}.

We assume
$$\pi:\Gamma\rightarrow \mathcal{U}({\mathcal H}),$$
is a unitary representation of $\Gamma$ on ${\mathcal H}$
and $\delta:{\mathcal H}\rightarrow{\mathcal H}$  a unitary operator verifying the following relation
\begin{equation}\label{Ba}
\delta^{-1}\pi(\gamma)\delta=\pi(\alpha(\gamma)),
\quad \text{for all}\; \gamma\in\Gamma.
\end{equation}
It follows that
\begin{equation}\label{Ba1}
\pi(\gamma)\delta^j=\delta\pi(\alpha(\gamma))\delta^{j-1}\dots=
\delta^j\pi(\alpha^j(\gamma)),\quad j\geq 1,
\quad \text{for all}\; \gamma\in\Gamma.
\end{equation}

For any given $\sigma$-finite measure $\mu$ on $\widehat{\Gamma},$
by $L^2(\widehat{\Gamma},\ell^2(\Gamma),\mu)$ we mean the Hilbert space of
(equivalence class of)  vector
functions $F$ defined on $\widehat{\Gamma}$, attaining values
in $\ell^2(\Gamma),$ which are measurable and square integrable with respect to the measure $\mu,$ i.e.
such that
$$\|F\|_2=\left( \int_{\widehat{\Gamma}}{\|F(\chi)\|_{\ell^2{(\Gamma)}}\; d\mu(\chi)}\right)^{1/2}<+\infty.$$

The scalar product in $L^2(\widehat{\Gamma},\ell^2(\Gamma),\mu)$ is given by
$$(F,G)=\int_{\widehat{\Gamma}}{(F(\chi),G(\chi))\; d\mu(\chi)},$$
where the inner product inside the integral is the one in $\ell^2{(\Gamma)}$.

We  recall the spectral theorem (see \cite[Theorem 4.44]{F}) 
and some of its consequences we shall need in the paper.
\begin{thm}\label{spe}
Let $\pi:\Gamma\rightarrow\mathcal{U}({\mathcal H}),$
be a unitary representation of a  locally compact abelian group $\Gamma$ on 
the Hilbert space ${\mathcal H}$.
Then there exists a projection valued measure $\Pi$ on $\widehat{G}$ such that
$$\pi(\gamma)=\int_{\widehat{G}}<\gamma,\chi>\, d\Pi(\chi),\quad 
\text{for all}\; \gamma\in\Gamma.$$
\end{thm}
 Combining the spectral theorem, Stone's theorem and the theory of spectral multiplicity, 
we have the following
decomposition theorem  for a  representation of a countable discrete locally compact abelian
 group
(see Weber's thesis \cite{We} for a proof)(see also \cite[Theorem on p.17]{HLN}):
\begin{thm}\label{web}
Let $\pi:\Gamma\rightarrow\mathcal{U}({\mathcal H}),$
be a unitary representation of a countable discrete locally compact abelian
 group $\Gamma$ on the Hilbert space 
${\mathcal H}$.
There exists a finite  measure $\mu$ on
Borel subsets of $\widehat{\Gamma}$ (we normalize it so that $\mu(\widehat{\Gamma})=1$)
and nested  measurable subsets
$$\dots\sigma_i\subset \dots\subset\sigma_2\subset\sigma_1\subset\widehat{\Gamma},$$
 there exists a unitary operator
\begin{equation}\label{ti}
T:{\mathcal H}\rightarrow \bigoplus_i{L^2(\sigma_i,\mu)}\hookrightarrow
L^2(\widehat{\Gamma},\ell^2(\Gamma),\mu),
\end{equation}
such that
\begin{itemize}
\item[1)]\begin{equation}\label{inter}
[T(\pi(\gamma)f)]_i(\chi)=(\gamma,\chi)[T(f)]_i(\chi),\end{equation}
for all $\gamma\in \Gamma,$
$f\in {\mathcal H},$
$\mu$-a.e. $\chi\in\widehat{\Gamma},$
where by $L^2(\sigma_i,\mu)$ we
mean $\mu$-measurable square summable (scalar) functions defined in $\widehat{\Gamma}$
 with support in $\sigma_i;$
\item[2)]
$T$ intertwines $\Pi$ with the canonical projection valued measure (given by multiplication by
characteristic function of Borel subsets),
i.e.\\  $T\circ {\Pi}(E)(f)=I_{E}\, T(f)$, for all $f\in {\mathcal H}$ and any Borel set
$E\subset\widehat{\Gamma}$.
\end{itemize}
The measure $\mu$ is unique up to equivalence of measures, and the $\sigma_i$'s are 
unique up to sets of $\mu$-measure $0.$
The  function $m:\widehat{\Gamma}\rightarrow \{0,1,\dots,+\infty\}$
 defined as 
$$
m(\chi)=\sharp\{\sigma_j,\;\chi\in\sigma_j\}=\sum_j{I_{\sigma_j}(\chi)},
$$
is called the multiplicity function.
\end{thm}

From \eqref{inter} we get the following identity for
$F\in L^2(\widehat{\Gamma},\ell^2(\Gamma),\mu),$  $f\in {\mathcal H},$ and  the inner
product in $\ell^2(\Gamma),$
\begin{equation}\label{ps}
(F(\chi),T(\pi(\gamma)f)(\chi))=\overline{(\gamma,\chi)}(F(\chi), T f(\chi)),
\end{equation}
for all  $\gamma\in \Gamma,$
$\mu$-a.e. $\chi\in\widehat{\Gamma}$.


We assume, since in general this is not the case, that $\mu$ is absolutely continuous
(which means, as is customery, absolutely continuous with respect to  Haar
measure $\lambda$ on $\widehat{\Gamma}$). Hence, by uniqueness of $\mu$ up to
equivalence of measures   and by uniqueness of the sets $\sigma_i$ up to sets of
$\mu$ measure $0$, we can assume that $\mu$ is the restriction of the Haar measure
$\lambda$ to the set $\sigma_1.$


A closed subspace $V\subset {\mathcal H}$ is said $\pi$-invariant if
$\pi(\gamma)V\subset V$, for all $\gamma\in \Gamma.$
We  use the following notation for a fixed $\psi\in {\mathcal H},$
$$Y=\{\delta^j\, \pi(\gamma)\,\psi, \;{j\in\Z},\, \gamma\in\Gamma\},$$
$$V_k=\overline{\text{span}}\{\delta^j \,\pi(\gamma)\,\psi,\; {j<k},\, \gamma\in\Gamma\}=\delta^k(V_0),\quad k\in\Z.$$
The indicator function of any set $A$ is denoted by $I_A.$ 
\\
All Hilbert spaces in this paper are separable.

\section{Extension of Bownik and Speegle result}\label{BS}
Results in this section extend some work by Bownik and Weber \cite{BoWe}, and Bownik and Speegle
\cite{BoSp} to abstract context. We include the proofs for sake of completeness.
%
\begin{defn}
The frame operator for a frame $(f_j)_{j\in J}$ in the Hilbert space ${\mathcal H}$ is
$$S:{\mathcal H}\rightarrow {\mathcal H}, \quad S(f)=\sum_{j\in J}{<f,f_j> f_j}.$$
It is a bounded, positive, invertible operator. The frame is a tight frame if and only if
$S=AI$, where $I$ is the identity operator. The frame is a Parseval frame iff $S=I$ i.e.
$$
\sum_{j\in J}{<f,f_j> f_j}=f, \quad\text{for all}\; f\in {\mathcal H}.
$$
\end{defn}
\begin{thm}\label{parse}
If $Y$ is a Parseval frame then for all $k\in\N$ the set
$$V_k=\overline{\textrm{span}}\,\{\delta^j\pi(\gamma)\,\psi, j<k, \gamma\in\Gamma\}$$
is $\pi$-invariant.
\end{thm}
\begin{proof}
Fix $k\geq 0.$ Since $Y$ is a Parseval frame, for all $f\in {\mathcal H},$ we have
\begin{eqnarray*}
f&=& \sum_{j\geq k}\sum_{\gamma\in \Gamma}{<f,\delta^j\pi(\gamma)\psi>\delta^j\pi(\gamma)\psi }\\ \\
& &+\sum_{j<k}\sum_{\gamma\in \Gamma}{<f,\delta^j\pi(\gamma)\psi>\delta^j\pi(\gamma)\psi } \\ \\
&=& B_1f+B_2f,
\end{eqnarray*}
where, by the frame property, the linear operators $B_i$ are bounded.

Now, if $\eta\in\Gamma,$ since $k$ is positive, by \eqref{Ba1},
\begin{eqnarray*}
\pi(\eta)B_1f&=&
\sum_{j\geq k}\sum_{\gamma\in \Gamma}
{<f,\delta^j\,\pi(\gamma)\,\psi>\pi(\eta)\delta^j\pi(\gamma)\,\psi }\\ \\
&=& \sum_{j\geq k}\sum_{\gamma\in \Gamma}
{<f,\delta^j\pi(\gamma)\,\psi>\delta^j\pi(\alpha^j(\eta))\pi(\gamma)\,\psi }.
\end{eqnarray*}
If we set $\nu=\alpha^j(\eta)\gamma$ (on the other hand any
$\nu$ can be written obviously as $\alpha^j(\eta)[\alpha^j(\eta)]^{-1}\nu$),
and we use \eqref{Ba1}, by unitariness of $\pi$  the latter is equal to
\begin{eqnarray*}
& &\sum_{j\geq k}\sum_{\nu\in \Gamma}
{<f,\delta^j\pi([\alpha^j(\eta)]^{-1})\pi(\nu)\,\psi>\delta^j\pi(\nu)\,\psi}\\ \\
&=&\sum_{j\geq k}\sum_{\nu\in \Gamma}
{<f,\pi(\eta)^{\ast}\delta^j\pi(\nu)\,\psi>\delta^j\pi(\nu)\,\psi }\\ \\
&=&\sum_{j\geq k}\sum_{\nu\in \Gamma}
{<\pi(\eta)f,\delta^j\pi(\nu)\,\psi>\delta^j\pi(\nu)\,\psi }\\ \\
&=&B_1(\pi(\eta)f).
\end{eqnarray*}
It follows that $B_2({\mathcal H})$ is $\pi$-invariant since
$$\pi(\eta)B_2f=\pi(\eta)(f-B_1f)=\pi(\eta)f-B_1(\pi(\eta)f)=B_2(\pi(\eta)f),$$
and so $\overline{B_2({\mathcal H})}$ is $\pi$-invariant as well.

Next we show that $\overline{B_2({\mathcal H})}=V_k,$ from which we obtain that $V_k$ is $\pi$-invariant.
 Indeed obviously ${B_2({\mathcal H})}\subset V_k$ and so
we get $\overline{B_2({\mathcal H})}\subset V_k.$ Conversely, if $f\in B_2({\mathcal H})^{\perp},$ then
$$0=<f,B_2f>=
\sum_{j<k}\sum_{\gamma\in \Gamma}{|<f,\delta^j\pi(\gamma)\psi>|^2 },$$
so $f\in V_k^{\perp}$ and everything is proved.
\end{proof}
\begin{lem}\label{lemma2}
Assume $V\subset{\mathcal H}$ is a  $\pi$-invariant closed subspace, and let $P_V$ be the
orthogonal projection onto $V.$ Then, for any $\gamma\in\Gamma,$ we have
$$P_V\pi(\gamma)=\pi(\gamma)P_V.$$
\end{lem}
\begin{proof}
Since $V$ is $\pi$-invariant, then $V^{\perp}$ is $\pi$-invariant, too.
%
Decomposing any $u\in{\mathcal H}$ as $u=v+w,$ where $v\in V$ and $w\in V^{\perp},$
we have
$$\pi(\gamma)P_Vu=\pi(\gamma)v=P_V(\pi(\gamma)v+\pi(\gamma)w)
=P_V(\pi(\gamma)(v+w))=P_V\pi(\gamma)u.$$
\end{proof}
\begin{thm}\label{primali}
Assume $V_0=\overline{\textrm{span}}\,\{\delta^j\pi(\gamma)\psi, j<0, \gamma\in\Gamma\}$
is $\pi$-invariant,  $V_0\neq V_1=\delta(V_0).$
Assume that for any $0\neq f\in{\mathcal H}$ the collection
$\{\pi(\gamma)f,\,\gamma\in\Gamma\}$ is linearly independent.

Then the affine system $Y$ is linearly independent.
\end{thm}
\begin{proof}
Suppose there exist a finite number of non-zero constants $c_{j,\gamma}\in\C$ such that
\begin{equation}\label{somma}
\sum_{j\in \mathbb{Z}}\sum_{\gamma\in\Gamma}{c_{j,\gamma}\,\delta^j\pi(\gamma)\psi}=0.
\end{equation}
By applying either $\delta$ or its inverse as many times as we need, we can suppose that the biggest $j$
in the sum \eqref{somma} such that $c_{j,\gamma}\neq 0$, for some $\gamma\in\Gamma,$ is $j=0.$

So \eqref{somma} leads to
$$\sum_{\gamma\in\Gamma}{c_{0,\gamma}\,\pi(\gamma)\psi}=
-\sum_{j<0}\sum_{\gamma\in\Gamma}{c_{j,\gamma}\,\delta^j\pi(\gamma)\psi}\in V_0.$$
If $P_{V_0}$ is the orthogonal projection onto $V_0,$ then, by Lemma \ref{lemma2}
\begin{eqnarray}
0&=&(I-P_{V_0})[\sum_{\gamma\in\Gamma}{c_{0,\gamma}\,\pi(\gamma)\psi}]\nonumber\\
&=&\sum_{\gamma\in\Gamma}{c_{0,\gamma}\,(I-P_{V_0})\pi(\gamma)\psi}\nonumber\\
&=&\sum_{\gamma\in\Gamma}{c_{0,\gamma}\,\pi(\gamma)(I-P_{V_0})\psi}.\label{pp}
\end{eqnarray}
Note that $\psi\notin V_0,$ otherwise, since $V_0$ is $\pi$-invariant, we get
$\pi(\gamma)\psi\in V_0$ for any $\gamma\in\Gamma,$  and  the contradiction
$$V_1=\overline{\textrm{span}}\,\{\delta^j\pi(\gamma)\psi, j<1, \gamma\in\Gamma\}=
\overline{\textrm{span}}\,\{\delta^j\pi(\gamma)\psi, j\leq 0, \gamma\in\Gamma\}=V_0.$$
Therefore $(I-P_{V_0})\psi\neq 0$, and  \eqref{pp} leads to a contradiction of
our hypothesis on linear independence.
\end{proof}

The above theorem obviously holds if we only assume
that for any $0\neq f\in V_0^{\perp}\subset{\mathcal H}$ the collection
$$\{\pi(\gamma)f, \gamma\in\Gamma\}$$
is linearly independent. A closer look to its proof shows
 that it
generalizes to more than one function, say $0\neq \psi_1,\dots,\psi_n\in
{\mathcal H},$
assuming that the set
$$\{\pi(\gamma)(I-P_{V_0})\psi_i, \gamma\in\Gamma, \,i=1,\dots,n\}$$
is linearly independent and no $\psi_i$ belongs to $V_0$.

This last remark is used in the following example, taken from \cite{BoSp}, to show
how the use of more general groups leads to results that cannot be reached just
using the group $\Z$.
\begin{ex}
Given $\varepsilon>0$, let us define  the function $\psi=\psi_0+\varepsilon \psi_1$
where
$$\hat{\psi_0}={\bf 1}_{[-1/4,-1/8]\cup[1/8,1/4]},\quad
\hat{\psi_1}={\bf 1}_{[-1/2,-1/4]\cup[1/4,3/4]}.$$
We note that the system $\{D_{2^j}T_k\psi\}$ is a frame  for sufficiently small $\varepsilon>0$, even if this does not matter here.

The space of negative dilates is
$$\{f\in\ld, \, \text{supp}\hat{f}\subset [-1/4,3/8],\, \hat{f}(\xi-1/2)=\hat{f}(\xi)\; \text{for a.e.}\; \xi\in[1/4,3/8]\}$$
which is $2\Z$-shift invariant but not  shift invariant.

Thus Bownik and Speegle's theorem does not apply, while
 a direct calculation shows that $\{D_{2^j}T_k\psi\}$ is linearly independent.

On the other hand, we note that
$$\{D_{2^j}T_k\psi, \;j,k\in\Z\}=\{D_{2^j}T_{{2k}}\phi, \;j,k\in\Z,
\;\phi=\psi, T_1\psi\},$$
the space of negative dilates being obviously the same $2\Z$-shift invariant space.
Furthermore both $\psi$ and $T_1\psi$ do not belong to $V_0.$

If we prove that the set
$\{T_{2k}(I-P_{V_0})\psi, T_{2k}(I-P_{V_0})T_1\psi,\, k\in\Z\}$ is
linearly independent, we can apply  the (generalization of the) above theorem with
$\Gamma=2\Z,$ $\pi(2k)=T_{2k}$  to conclude
that $\{D_{2^j}T_k\psi, j,k\in\Z\}$ is linearly independent.

Now an easy calculation shows that
$(I-P_{V_0})\psi$ and $(I-P_{V_0})T_1\psi$ have Fourier transform,
respectively equal to
$$
\varepsilon{\bf 1}_{[-1/2,-1/4]\cup[3/8,3/4]}
+(\frac{1-\varepsilon}{2})
({\bf 1}_{[-1/4,-1/8]}-{\bf 1}_{[1/4,3/8]})$$
and
$$
\varepsilon e^{-2\pi i\xi}{\bf 1}_{[-1/2,-1/4]\cup[3/8,3/4]}
+(\frac{1+\varepsilon}{2})e^{-2\pi i\xi}
{\bf 1}_{[-1/4,-1/8]\cup[1/4,3/8]}.
$$
Hence, since the intervals $[-1/4,-1/8]$ and $[1/4,3/8]$ have disjoint
intersection with $[-1/2,-1/4]\cup[3/8,3/4],$ the linear independence follows.
\end{ex}

\section{Invariant spaces and range functions}\label{inv}
The purpose of this section is to  provide a version of Helson's theorem,
\cite[Theorem 8]{H64},
adapted to $\pi$-invariant spaces. In the case of translations on $L^2(G)$, a proof can be found in
the work by Bownik \cite{Bo} for uniform lattices in $\R^n$,  Cabrelli and Paternostro \cite{CaPa}
for uniform lattices in LCA group, and Bownik and Ross \cite{BoRo} for co-compact,
but not necessarily discrete, subgroups.
We stress that, in the opposite,  our generalization applies to discrete, but
not necessarily co-compact, subgroups, and that we assume the measure $\mu$ to be
absolutely continuous. Moreover the subsequent Corollary \ref{coro1} will be of fundamental
importance in the proof of the main result in Section \ref{main}.

%
%
\begin{defn} Assume $\mu$ is a $\sigma$-finite measure on $\widehat{\Gamma}.$
A range function is any map
$$J:\widehat{\Gamma}\rightarrow \{\text{closed subspaces of}\; \ell^2(\Gamma)\}.$$
$J$ is said measurable if, denoted by $P(\chi)$ the orthogonal projection
onto $J(\chi)$, for all $a,b\in \ell^2(\Gamma),$ the map
$$\chi\in\widehat{\Gamma}\mapsto(P(\chi)a,b)\in\C, $$
is $\mu$-measurable.
\end{defn}
Range functions are identified if they are a.e. equal with respect to the measure
$\mu$ on $\widehat{\Gamma}.$
\begin{defn}\label{defmj}
Let $J$ be a range function, let $\mu$ be a $\sigma$-finite measure on $\widehat{\Gamma}.$
We define
\begin{equation}\label{mj}
M_J=\{ F\in L^2(\widehat{\Gamma},\ell^2(\Gamma),\mu),\; F(\chi)\in J(\chi),\; \mu\text{-a.e. }
\chi\in\widehat{\Gamma}\}.
\end{equation}
\end{defn}
\begin{rem}\label{rem1}\cite[p. 57]{H64}\cite[p.6, ex.2]{HLN}

It is useful to recall that, for every sequence
$F_n\in L^2(\widehat{\Gamma},\ell^2(\Gamma),\mu),$ $n\in\N,$ converging
to $F\in L^2(\widehat{\Gamma},\ell^2(\Gamma),\mu)$ in norm, there exists a subsequence
$F_{n_k}$, $k\in\N,$ converging   to $F(\chi),$ pointwise a.e..

It follows that $M_J$ is a closed subspace of $L^2(\widehat{\Gamma},\ell^2(\Gamma),\mu).$
\end{rem}

The next lemma, proved in \cite[p.58]{H64} for $\widehat{\Gamma}=\mathbb{T}$, extends, mutatis mutandis, to
the general setting.

\begin{lem}\label{H64p58}
Let $J$ be a measurable range function. Let $M_J$ be the space defined  by \eqref{mj}.
Let $$\mathscr{P}:L^2(\widehat{\Gamma},\ell^2(\Gamma),\mu)\rightarrow L^2(\widehat{\Gamma},\ell^2(\Gamma),\mu)$$
be the orthogonal projection onto $M_J$ and,
for any $\chi\in \widehat{\Gamma}$, let us denote by
$P(\chi):\ell^2(\Gamma)\rightarrow \ell^2(\Gamma)$ the
 orthogonal projection onto $J(\chi).$ Then, for any $F\in L^2(\widehat{\Gamma},\ell^2(\Gamma),\mu),$
we have
\begin{equation}\label{pf}
(\mathscr{P}F)(\chi)=P(\chi)F(\chi),\quad \mu-\text{a.e.}\; \chi\in \widehat{\Gamma}.
\end{equation}
\end{lem}
Consider the unitary operator  $T$ given in Theorem \ref{web},\eqref{ti}.
%
%
%

The next theorem, which generalizes Helson's theorem, says that the same $T$ maps
unitarily  any $\pi$ invariant subspace $V$ onto a certain $M_{J_V}$.

We need a preliminary lemma. In the terminology of \cite{BoRo}, it states that
 ${\Gamma}$ is a determining set for
$L^1(\widehat{\Gamma}, \mu).$
\begin{lem}\label{comp}
Let $g:\widehat{\Gamma}\rightarrow \C$ be in $L^1(\widehat{\Gamma}, \mu),$
 such that, for all $\gamma\in \Gamma,$
$$
\int_{\widehat{\Gamma}}{(\gamma,\chi)\, g(\chi)\; d\mu(\chi)}=0,
$$
where $\mu$ is a finite measure on Borel sets of $\widehat{\Gamma}$ which is absolutely
continuous.
Then for $\mu$-almost all $\chi\in \widehat{\Gamma},$ we have $g(\chi)=0.$
\end{lem}
\begin{proof}
Let $h\in
L^1(\widehat{\Gamma},\lambda),$ be given by the Radon-Nikod\'{y}m theorem, such that $\mu(E)=\int_E h(\chi) \,d\lambda(\chi),$ and
$$\int_{\widehat{\Gamma}}{(\gamma,\chi)\, g(\chi) h(\chi)\; d\lambda(\chi)}=0,
\quad \text{for all}\, \gamma\in\Gamma.$$
By Pontryagin duality theorem and Fourier uniqueness theorem, we get,
for $\lambda$ almost all $\chi\in \widehat{\Gamma},$ $g(\chi) h(\chi)=0.$

Let us denote by $A\subset\widehat{\Gamma}$ the set where $g(\chi) h(\chi)\neq0,$
and by $B\subset\widehat{\Gamma}$ the set where $h(\chi)=0$. Then $\lambda(A)=0$
and so $\mu(A)=0$ by absolute continuity.  Also
$\mu(B)=\int_B h(\chi)\,d\lambda(\chi)=0.$ But
$$g(\chi)\neq 0\Longrightarrow \chi\in A\cup B,$$
so the result follows.
\end{proof}

\begin{thm}\label{hel}
Let $V\subset {\mathcal H}$ be a $\pi$-invariant closed subspace, where
$\pi:\Gamma\rightarrow\mathcal{U}({\mathcal H})$ is a unitary representation. Then
\begin{equation}\label{setV}
V=\{f\in{\mathcal H}, \; Tf(\chi)\in J_V(\chi),\; \mu\text{-a.e.}\, \chi\in\widehat{\Gamma}\},
\end{equation}
where $T$ is the unitary operator in \eqref{ti}  and $J_V$ is a $\mu$-measurable range function.
The correspondence between $V$ and $J_V$ is one-to-one. Moreover
\begin{equation}\label{vspan}
V=\overline{\text{span}}\{\pi(\gamma)\varphi,\; \gamma\in \Gamma, \varphi\in \mathscr{A}\},
\end{equation}
for an at most countable set $\mathscr{A}\subset{\mathcal H},$  and for any such $\mathscr{A}$ verifying
\eqref{vspan} we have,
$$J_V(\chi)=\overline{\text{span}}\{T \varphi(\chi),\;  \varphi\in \mathscr{A}\},\quad
\mu\text{-a.e.}\,\chi\in\widehat{\Gamma}.$$
\end{thm}
\begin{proof}
 We omit the incessant reference to the measure $\mu$, hence
 it is assumed in this proof that a.e. means $\mu$-a.e..

In order to prove \eqref{setV}, we need to show that
$$T(V)=\{ F\in L^2(\widehat{\Gamma},\ell^2(\Gamma),\mu),\; F=Tf,\, f\in V\}=M_{J_V},$$
for a suitable measurable range function $J_V$.
Indeed, if the latter is true,
$f\in V$ implies $T f\in T(V)=M_{J_V}$, which means, by definition, that
$ T f(\chi)\in J_V(\chi),$ a.e..
Conversely, if $T f(\chi)\in J_V(\chi)$, a.e., then
$ T f\in M_{J_V}=T(V)$ and so $ T f=T g$
for some $g\in V$, yielding $f=g,$
since $T$ is one to one.

Once we prove $T(V)=M_{J_V}$, the uniqueness of $J_V$ comes from Lemma \ref{H64p58}.
Indeed, assume $T(V)=M_{J_V}=M_K$ for two measurable range functions. Let $\mathscr{P}$
be the orthogonal projection onto $T(V)$ and $P(\chi)$, $Q(\chi)$ be the orthogonal
projections onto $J_V(\chi)$ and $K(\chi)$ respectively. Then Lemma \ref{H64p58}
implies that for a.e. $\chi\in \widehat{\Gamma},$ and for all $F\in
 L^2(\widehat{\Gamma},\ell^2(\Gamma),\mu)$
$$P(\chi)F(\chi)=(\mathscr{P} F)(\chi)= Q(\chi)F(\chi).$$
In particular, for any $a\in \ell^2(\Gamma),$
$$P(\chi)a=P(\chi)P(\chi)a= Q(\chi)P(\chi)a,\quad\text{and}\quad
Q(\chi)a= P(\chi)Q(\chi)a,$$
hence the range of $P(\chi)$ equals the range of $Q(\chi)$ that means a.e.
$J_V(\chi)=K(\chi)$, i.e. $J_V=K$.

Now we prove \eqref{vspan}.

Let $\mathcal{E}$ be an orthonormal basis for $\ell^2(\Gamma).$
Note that $\mathcal{E}$ is countable since $\ell^2(\Gamma)$  is separable.

Let us consider the
following  elements in $L^2(\widehat{\Gamma},\ell^2(\Gamma),\mu)$,
$$F_{\gamma,e}(\chi)=(\gamma,\chi) \,e\in \ell^2(\Gamma),\quad
\gamma\in\Gamma, e\in\mathcal{E}.$$
We prove that
$\overline{\text{span}}\{F_{\gamma,e}, \gamma\in\Gamma, e\in\mathcal{E}\}=
L^2(\widehat{\Gamma},\ell^2(\Gamma),\mu)$.

Indeed,
if $F\in L^2(\widehat{\Gamma},\ell^2(\Gamma),\mu)$ is  such that
$0=(F_{\gamma,e},F)$ for  all $\gamma\in\Gamma, e\in\mathcal{E},$ then
$$
0=(F_{\gamma,e},F)=\int_{\widehat{\Gamma}}{(F_{\gamma,e}(\chi),F(\chi))\, d\mu(\chi)}
=\int_{\widehat{\Gamma}}{(\gamma,\chi) (e, F(\chi))\, d\mu(\chi)}.
$$
It follows that  the $L^1$ function ($e$ fixed)
$$\chi\in \widehat{\Gamma}\mapsto (e, F(\chi))\in\C,$$
verifies the hypotheses of Lemma \ref{comp}, hence
$(e, F(\chi))=0$ a.e. $\chi\in \widehat{\Gamma},$ for all $e\in\mathcal{E}.$ Since
$\mathcal{E}$ is an orthonormal basis of $\ell^2(\Gamma)$ we have $F(\chi)=0$ a.e.
$\chi\in \widehat{\Gamma},$
and so $F\equiv 0,$ as desired.

Let ${\cal{P}}(F_{\gamma,e})$ be the projection onto $T(V)$, then
${\cal{P}}(F_{\gamma,e})=T\varphi_{\gamma,e},$ for some $\varphi_{\gamma,e}\in V.$

By above, the set
$\{{\cal{P}}(F_{\gamma,e})=T\varphi_{\gamma,e},\, \gamma\in\Gamma,e\in\mathcal{E}\}$
 spans the range of ${\cal{P}}$, i.e.
$T(V).$

We claim that
$$V=\overline{\text{span}}\{\pi(\eta)\varphi_{\gamma,e},\;
\eta,\gamma\in \Gamma, e\in\mathcal{E}\}.$$
Indeed, since $V$ is $\pi$-invariant, it is obvious that
$\pi(\eta)\varphi_{\gamma,e}\in V,$ so that
$$\overline{\text{span}}\{\pi(\eta)\varphi_{\gamma,e},\;
\eta,\gamma\in \Gamma, e\in\mathcal{E}\}\subset V.$$
On the other hand, if  $f\in V$ such that $0=(f,\pi(\eta)\varphi_{\gamma,e})$,
for all $\eta,\gamma\in \Gamma, e\in\mathcal{E},$ since $T$ is a unitary map,
 by \eqref{ps}
\begin{eqnarray*}
0&=&(Tf,T(\pi(\eta)\varphi_{\gamma,e}))=
\int_{\widehat{\Gamma}}{(Tf(\chi),T(\pi(\eta)\varphi_{\gamma,e})(\chi))\, d\mu(\chi)}\\
&=&\int_{\widehat{\Gamma}}{\overline{(\eta,\chi)}(Tf(\chi),{\cal{P}}(F_{\gamma,e})(\chi))\, d\mu(\chi)}.
\end{eqnarray*}
Again, by Lemma \ref{comp} (${\gamma,e}$ fixed),
 $(Tf(\chi),{\cal{P}}(F_{\gamma,e})(\chi))=0$ a.e. $\chi\in \widehat{\Gamma},$ for all
$\gamma\in \Gamma, e\in\mathcal{E}$. Hence
$$(Tf,{\cal{P}}(F_{\gamma,e}))=
\int_{\widehat{\Gamma}}{(Tf(\chi),{\cal{P}}(F_{\gamma,e})(\chi))\, d\mu(\chi)}=0.$$
It follows, since
$\{{\cal{P}}(F_{\gamma,e})=T\varphi_{\gamma,e},\, \gamma\in \Gamma, e\in\mathcal{E}\}$ spans
${\cal{P}}(L^2(\widehat{\Gamma},\ell^2(\Gamma),\mu)=T(V),$ that $f\equiv 0$,
and so the claim is proved. Hence \eqref{vspan}
is proved with $\mathscr{A}$ being the (countable) collection of all $\varphi_{\gamma,e}.$
\vspace{.3cm}

Next let $\mathscr{A}$ be an at most countable set verifying \eqref{vspan}. Let us define the range
function $J_V$ as
$$J_V(\chi)=\overline{\text{span}}\{T \varphi(\chi),\;  \varphi\in \mathscr{A}\}\subset \ell^2(\Gamma)
,\quad
\text{a.e.}\,\chi\in\widehat{\Gamma}.$$
We show that $T(V)=M_{J_V}.$

If $F\in T(V)$, let $f\in V$ such that $Tf=F.$ Taken a sequence
$f_n\in {\text{span}}\{\pi(\gamma)\varphi,\; \gamma\in \Gamma, \varphi\in\mathscr{A}\}$ such that
$f_n\rightarrow f$ in norm, it follows that $Tf_n\rightarrow Tf=F.$

Now, for any $\gamma\in \Gamma,$ $\varphi\in \mathscr{A},$ and a.e. $\chi\in\widehat{\Gamma},$
\eqref{inter} implies
$$T(\pi(\gamma)\varphi)(\chi)=(\gamma,\chi) T\varphi(\chi)\in J_V(\chi),$$
since $(\gamma,\chi)\in\C.$ So a.e. also $Tf_n(\chi)\in J_V(\chi).$

As we pointed in Remark \ref{rem1}, there exists a subsequence such that
$Tf_{n_k}\rightarrow F$ a.e.. Since $J_V(\chi)$ is closed, this implies that
$F(\chi)\in J_V(\chi),$ a.e. and so $F\in M_{J_V}$ as required.

Conversely, assume that $T(V)\subsetneqq M_{J_V}.$ Then, there exists a non zero $F\in M_{J_V},$
such that $F\in T(V)^{\perp}.$ This yields, for all $\gamma\in \Gamma,$ and $\varphi\in\mathscr{A},$
\begin{eqnarray*}
0&=&\int_{\widehat{\Gamma}}{(T(\pi(\gamma)\varphi)(\chi),F(\chi))\, d\mu(\chi)}\\
&=&\int_{\widehat{\Gamma}}{(\gamma,\chi)(T\varphi(\chi),F(\chi))\, d\mu(\chi)},
\end{eqnarray*}
and the same reasoning above yields  $(T\varphi(\chi),F(\chi))=0$
for all $\varphi\in\mathscr{A},$ and a.e. $\chi\in\widehat{\Gamma}.$
In particular, since $F(\chi)\in J_V(\chi), $  we get
$(F(\chi),F(\chi))=0$ a.e., which implies the contradiction $\|F\|^2=0.$

Finally, it remains to show that $J_V$ is measurable.
Let $\mathscr{P}$ be the orthogonal projection onto $T(V)=M_{J_V},$ and $P(\chi)$
the orthogonal projection onto $J_V(\chi).$
By Lemma \ref{H64p58}, if we take the constant function $F(\chi)=a\in \ell^2(\Gamma),$
for a fixed $a$, by \eqref{pf} and every $b\in \ell^2(\Gamma),$
$$(P(\chi)a,b)=(P(\chi)F(\chi),b)=((\mathscr{P} F)(\chi),b).$$
The function $\chi\mapsto ((\mathscr{P} F)(\chi),b)$ is measurable since
$\mathscr{P} F$ is, so for all $a,b\in \ell^2(\Gamma)$ the function
$\chi\mapsto (P(\chi)a,b)$ is measurable and everything is proved.
\end{proof}

\begin{coro}\label{coro1}
Let $V\subset{\mathcal H}$ be a closed $\pi$-invariant subspace, $\varphi\in V,$ and
$f\in{\mathcal H}.$ If $T$ denotes the unitary operator in \eqref{ti}, suppose that for $\mu$-almost all
$\chi\in \widehat{\Gamma}$ there exists a constant
$c(\chi)\in\C$ such that, for all $i$,
\begin{equation}\label{ts}
[T f]_i(\chi)=c(\chi) [T \varphi]_i(\chi).
\end{equation}
 Then $f\in V.$
\end{coro}
\begin{proof}Since $V$ is $\pi$-invariant, for all $\gamma\in\Gamma$
we have $\pi(\gamma)\varphi\in V,$ and
$$S:=\overline{\text{span}}\{\pi(\gamma)\varphi,\; \gamma\in\Gamma\}\subset V.$$
Hence it suffices to prove that $f\in S.$

$S$ is $\pi$-invariant so, by Theorem \ref{hel}, we can write it in terms of its range
function
\begin{equation}\label{esse}
S=\{g\in\ldg, \; T g(\chi)\in J_S(\chi), \; \mu{\text-a.e.}\; \chi\in \widehat{\Gamma}\},
\end{equation}
where, since $S$ is generated by $\varphi$,
$$J_S(\chi)=\overline{\text{span}}\{T\varphi(\chi)\}=\{\lambda\, T\varphi(\chi),\;
\lambda\in\C\}.$$
\\
But \eqref{ts} implies that, for $\mu$-a.e. $\chi$, $T f(\chi)\in J_S(\chi)$ and so $f\in S$
by \eqref{esse}.
\end{proof}
\begin{defn}
Let $V\subset{\mathcal H}$ be a $\pi$-invariant closed space and let $J_V$ be a range
function associated with $V$ as in \eqref{setV} of Theorem \ref{hel}. If
$\mathscr{A}$ is a countable set verifying \eqref{vspan}, we define, for $\mu$-almost all
$\chi\in\widehat{\Gamma},$
$$
\dim_V(\chi)=\dim J_V(\chi)=\dim\overline{\text{span}}\{T\psi(\chi),\,
\psi\in\mathscr{A}\},
$$
where the latter means dimension as a vector subspace in $\ell^2(\Gamma).$
 If $\mathscr{A}$ is finite, we say that $V$ is finitely generated.
 If $V={\mathcal H}$ we simply write $\dim J(\chi).$

An elementary argument about  vector spaces is at the basis of the following proposition.
\end{defn}
\begin{prop}\label{dim}
Let $V\subset{\mathcal H}$ be a $\pi$-invariant closed space and let $J_V$ be a range
function associated with $V$ as in \eqref{setV} of Theorem \ref{hel}.
Let $\mathscr{A}$ be a countable set verifying \eqref{vspan}. If
$V$ is finitely generated, then, for $\mu$-almost all
$\chi\in\widehat{\Gamma},$ $\dim J_V(\chi)\leq\sharp\mathscr{A}<+\infty.$
\end{prop}

\section{The multiplicity function}\label{smult}
This section is devoted to the multiplicity function and some basic formulas for it.
In particular we establish in Proposition \ref{dimult} a link between the dimension function of
an invariant $\pi$-subspace and the multiplicity function of the subrepresentation 
generated by it.

As we have already seen, any time we are given a unitary representation
$$\pi:\Gamma\rightarrow \mathcal{U}({\mathcal H}),$$
we have  a Borel measure $\mu$ on $\widehat{\Gamma},$  a
 unitary operator $T$ defined as in \eqref{ti}, and  an associated
multiplicity function  $m$.

Assume now $V\subset{\mathcal H}$ is invariant under $\pi$. It determines, in an obvious way, a unitary
representation
$$\widetilde{\pi}:\Gamma\rightarrow \mathcal{U}(V),\quad
\widetilde{\pi}(\gamma)f={\pi}(\gamma)f, \quad f\in V,$$
called the subrepresentation of $\pi$ on $V$.

As above, we get a Borel measure $\widetilde{\mu}$ on $\widehat{\Gamma}$,
and measurable subsets
\begin{equation}\label{sigma1}
\dots\widetilde{\sigma}_i\subset \dots\subset\widetilde{\sigma}_2\subset\widetilde
{\sigma}_1\subset\widehat{\Gamma},
\end{equation}
 a unitary map
$
\widetilde{T}:V\rightarrow \displaystyle\bigoplus_i{L^2(\widetilde{\sigma}_i,\widetilde{\mu})}\hookrightarrow
L^2(\widehat{\Gamma},\ell^2(\Gamma),\widetilde{\mu}),
$
such that
\begin{equation}\label{intertitilde}
[\widetilde{T}(\widetilde{\pi}(\gamma)f)]_i(\chi)=(\gamma,\chi)[\widetilde{T}(f)]_i(\chi),\quad\text{for all}\;  \gamma\in \Gamma, f\in V,
\;\widetilde{\mu}\,\text{a.e.}\; \chi\in\widehat{\Gamma},
\end{equation}
and a multiplicity function, denoted by $\widetilde{m}.$

Furthermore, we denote by $\widetilde{J}$ the range function provided by 
 Theorem \ref{hel}, such that
$\widetilde{T}:V\rightarrow M_{\widetilde{J}}$.
 We shall always affix the {\em{tilde}} to objects
 related to  subrepresentations of $\pi$.

\begin{rem}\label{rap}
Note that, since $\widetilde{\pi}$ is the subrepresentation of $\pi$ on $V$,
the measure
$\widetilde{\mu}$  is absolutely continuous with respect to $\mu$
(and hence absolutely continuous); we
can also  prove it directly, and we choose to do so, since this allows to
introduce some elements of spectral theory that we will need later.

We recall that, by the
spectral theorem, see \cite{HLN}, representations $\pi$ and $\widetilde{\pi},$
are linked respectively to the projection valued measures $\Pi$ and $\widetilde{\Pi},$
in the following way:
\begin{equation}\label{spectra}
\pi(\gamma)=\int_{\widehat{\Gamma}}{(\gamma,\chi)\,d\Pi(\chi)},\quad
\widetilde{\pi}(\gamma)=\int_{\widehat{\Gamma}}{(\gamma,\chi)\,d\widetilde{\Pi}(\chi)}.
\end{equation}

The meaning of \eqref{spectra} is that for any $f,g\in{\mathcal H}$,
and $f',g'\in V,$ we have
\begin{equation}\label{value}
(\pi(\gamma)f,g)=\int_{\widehat{\Gamma}}{(\gamma,\chi)\,d m_{f,g}(\chi)},\quad
(\widetilde{\pi}(\gamma)f',g')=\int_{\widehat{\Gamma}}{(\gamma,\chi)\,d\widetilde{m}_{f',g'}(\chi)},
\end{equation}

where  measures $m_{f,g}$ and $\widetilde{m}_{f',g'}$ are defined by
\begin{equation}\label{meas}
m_{f,g}(E)=(\Pi(E)f,g),\quad \widetilde{m}_{f',g'}(E)=(\widetilde{\Pi}(E)f',g'),
\end{equation}
for all Borel sets $E\subset\widehat{\Gamma}.$

An application of Zorn's lemma yields that there exist $f\in{\mathcal H}$ and $g\in V$ such that,
for any Borelian $E\subset\widehat{\Gamma}$,
$$\mu(E)=m_{f,f}(E)=({\Pi}(E)f,f),\quad \widetilde{\mu}(E)=\widetilde{m}_{g,g}(E)=(\widetilde{\Pi}(E)g,g).$$

Moreover, for any other $h\in\ldg$ we have $m_{h,h}\ll \mu$, and,
since $\widetilde{\pi}$ is the subrepresentation of $\pi$ on $V$, we have also, by  uniqueness
of Fourier-Stieltjies transform,
$$\widetilde{\mu}(E)=(\widetilde{\Pi}(E)g,g)=({\Pi}(E)g,g),\quad \text{since}\; g\in V.$$
Finally, if $\mu(E)=0$ then $m_{g,g}(E)=0$ and, by above, $\widetilde{\mu}(E)=0.$
\end{rem}

An additional consequence of Theorem \ref{web} is the following (see \cite{Ba})
\begin{lem}\label{lemmaS}
Suppose that $\pi$ is a unitary representation of the abelian group $\Gamma$  acting
on a Hilbert space $\mathcal{H}$, and let $\nu$ and $\tau_i$ be, respectively, the
Borel measure and the Borel measurable sets as in Theorem \ref{web}.

Suppose that
$\tau'_j$ is another collection of, not necessarily nested, Borel subsets of
$\widehat{\Gamma}$, and  $T'$ is a unitary operator,
$$T':\mathcal{H}\rightarrow \bigoplus_j L^2(\tau'_j,\nu)$$ satisfying
$$[{T'}({\pi}(\gamma)f)]_j(\chi)=(\gamma,\chi)[{T'}(f)]_j(\chi),\quad\text{for all}\;  \gamma\in \Gamma, f\in \mathcal{H},
\;{\nu}\,\text{a.e.}\; \chi\in\widehat{\Gamma}.$$
Then, for $\nu-$almost all $\chi\in\widehat{\Gamma},$
$$\sum_i{I_{{\tau}_i}(\chi)}=\sum_j{I_{\tau'_j}(\chi)}.$$
\end{lem}

\begin{rem}
We recover, by the  lemma above, that  the multiplicity function of the
representation $\pi$ on ${\mathcal H}$ satisfies
 $$m(\chi)=\dim{{J}(\chi)},\quad\mu\text{-a.e}\;\chi\in\widehat{\Gamma},$$
if $J$ is the range function associated with the operator $T$ and ${\mathcal H}$.

 Indeed it suffices to take $T'=T,$ $\tau_i=\sigma_i,$ and $\tau'_j$ being the set
 where $\dim{{J}(\chi)}=j$.

 The same argument applies for a subrepresentation on a $\pi$-invariant subspace as well.

\end{rem}

The notion of Borel cross-section will be useful for us.
\begin{rem}\label{cross}
A Borel cross-section for the quotient map
$q:\widehat{\Gamma}\rightarrow\widehat{\Gamma}/\ker{\alpha^{\ast}}$ is a
Borel measurable  right inverse for $q,$  i.e. a map
$\widetilde{s}:\widehat{\Gamma}/\ker{\alpha^{\ast}}\rightarrow \widehat{\Gamma}$
 such that
$q\circ\widetilde{s}=\text{Id}_{\widehat{\Gamma}/\ker{\alpha^{\ast}}}.$

Since $\ker{\alpha^{\ast}}$ is closed, and $\widehat{\Gamma}$ is compact and metrizable (hence separable),
a Borel cross-section for $q$ exists by Mackey's result in \cite[Lemma 1.1]{Ma}.

It follows that there exists a measurable map $s:\widehat{\Gamma}\rightarrow\widehat{\Gamma},$ such that
$\alpha^{\ast}(s(\chi))=\chi.$

%
\end{rem}

In the sequel we shall need the following formula, contained implicitly in \cite{Ba},
which generalizes the analogous  formula obtained by Bownik and Rzeszotnik, in Corollary 2.5 of \cite{BoRz},
for shift invariant spaces in $L^2(\R^n).$
\begin{lem}\label{formulaB}
Assume we are given a closed subspace $V\subset{\mathcal H},$ and let $W=\delta(V)$, where
$\delta:{\mathcal H}\rightarrow {\mathcal H}$ is a unitary map verifying \eqref{Ba}.

Suppose $V$ is $\pi$-invariant.
Let us denote by $\widetilde{\mu},$ $\widetilde{T},$ and $\widetilde{m}$ the usual objects given by Theorem \ref{web}
related to the subrepresentation of $\pi$ on $V$ and by ${\mu^{\sharp}},$ ${T}^{\sharp},$ and ${m}^{\sharp}$
the corresponding objects for $W$. Then we have, for ${\mu}^{\sharp}$-almost
all $\chi\in\widehat{\Gamma},$
\begin{equation}\label{formulaBa}
{m}^{\sharp}(\chi)=\sum_{\alpha^{\ast}(\xi)=\chi}{\widetilde{m}(\xi)}.
\end{equation}
\end{lem}
\begin{proof}
Let $s:\widehat{\Gamma}\rightarrow \widehat{\Gamma}$ be the map linked to the Borel cross-section for the quotient map
$q:\widehat{\Gamma}\rightarrow\widehat{\Gamma}/\ker{\alpha^{\ast}},$
as in Remark \ref{cross}.

For any index $i$ and any $\eta\in\ker{\alpha^{\ast}}$(recall that
$|\ker{\alpha^{\ast}}|=N$), let us define
a collection of Borel measurable sets by
$$\tau_{i,\eta}=\{\chi\in\widehat{\Gamma},\; s(\chi)\,\eta\in\widetilde{\sigma}_i\},$$
and the map
$$T':\delta(V)\rightarrow\bigoplus_{i,\eta}{L^2(\tau_{i,\eta},{\mu}^{\sharp})},$$
by
$$[T'(f)]_{i,\eta}(\chi)=[\widetilde{T}(\delta^{-1}(f))]_i(s(\chi)\eta),\quad
\chi\in\tau_{i,\eta}.$$
We have by \eqref{Ba}, and \eqref{intertitilde}, for $f\in \delta(V),$ $\chi\in\tau_{i,\eta},$
\begin{eqnarray*}
[T'(\pi(\gamma)f)]_{i,\eta}(\chi)&=&[\widetilde{T}(\delta^{-1}(\pi(\gamma)f))]_i\,(s(\chi)\,\eta)\\
&=&[\widetilde{T}(\pi(\alpha(\gamma))\delta^{-1}(f))]_i\,(s(\chi)\,\eta)\\
&=&(\alpha(\gamma),s(\chi)\,\eta)\,[\widetilde{T}(\delta^{-1}(f))]_i\, (s(\chi)\,\eta)\\
&=&(\gamma,\alpha^{\ast}(s(\chi)\,\eta))\,[T'(f)]_{i,\eta} (\chi)\\
&=&(\gamma,\chi)\,[T'(f)]_{i,\eta} (\chi).
\end{eqnarray*}
Hence, by Lemma \ref{lemmaS}, the multiplicity function ${m}^{\sharp}$ is given,
${\mu}^{\sharp}$-a.e. $\chi\in\widehat{\Gamma},$ by
\begin{eqnarray*}
{m}^{\sharp}(\chi)&=&\sum_{i,\eta}{{I}_{\tau_{i,\eta}}(\chi)}
=\sum_{\eta}\sum_i{{I}_{\widetilde{\sigma}_{i}}(s(\chi)\,\eta)}\\
&=&\sum_{\eta}{\widetilde{m}(s(\chi)\,\eta)}
=\sum_{\alpha^{\ast}(\xi)=\chi}{\widetilde{m}(\xi)}.
\end{eqnarray*}
\end{proof}

Any time we have a $\pi$-invariant subspace $V$, we are given two range functions:
 $J_V$ by Theorem \ref{hel}, and $\widetilde{J}$. It
is worth to compare the respective dimensions  $\dim J_V(\chi)$ and
$\dim \widetilde{J}(\chi):$
in the following proposition we use the fact that $\widetilde{\mu}$ is absolutely
continuous with respect to $\mu$ to show that a relation always exists.

\begin{prop}\label{dimult}
 Let $V\subset{\mathcal H}$ be a
$\pi$-invariant subspace, and $J_V$ be the range function associated with $V$, provided
by Theorem \ref{hel}. Consider  the subrepresentation
of $\pi$ on $V$, $\widetilde{\pi}:\Gamma\rightarrow\mathcal{U}(V)$, the
 range function $\widetilde{J}$ associated with the unitary map $\widetilde{T},$
and the multiplicity function $\widetilde{m}.$

Then for  $\widetilde{\mu}$ almost all $\chi\in\widehat{\Gamma}$,  $\widetilde{m}(\chi)=\dim \widetilde{J}(\chi)\leq \dim{J_V(\chi)}.$
\end{prop}
\begin{proof}
By Helson's theorem, Theorem \ref{hel}, we get
${T}(V)=M_{{J_V}},$ and $\widetilde{T}(V)=M_{{\widetilde{J}}}.$
The composition $T_{|V}\circ \widetilde{T}^{-1}:M_{\widetilde{J}}\rightarrow
M_{{J_V}}$ is a unitary map. We call it $T\circ \widetilde{T}^{-1}$ for short.

Consider $n\in\N$, and the  set $\widetilde{\tau}_n$ where
$$\widetilde{m}(\chi)= \dim \widetilde{J}(\chi)=n.$$

 Let $F_1,\dots,F_n\in M_{\widetilde{J}},$ be pointwise orthonormal
$\widetilde{\mu}$ a.e. on $\widetilde{\tau}_n$  and vanishing out of it (see \cite{HLN} p.12
problem 4).
We aim to show that pointwise orthonormality of
$F_1(\chi),\dots,F_n(\chi)\in\widetilde{J}(\chi)$
implies pointwise orthonormality of a same number of elements in $J_V(\chi).$
By \eqref{ti} we have
\begin{eqnarray}
& &\int_{\widehat{\Gamma}}{(\gamma,\chi)\,
(T\circ\widetilde{T}^{-1}(F_j)(\chi),T\circ\widetilde{T}^{-1}(F_h)(\chi))
\; d\mu(\chi)}\label{point}\\
&=&
\int_{\widehat{\Gamma}}{
(T(\pi(\gamma)\widetilde{T}^{-1}F_j)(\chi),T(\widetilde{T}^{-1}F_h)(\chi)
)\; d\mu(\chi)}.\nonumber
\end{eqnarray}

Since $T$ is a unitary operator, by recalling the definition of inner product in the vector space
$L^2(\widehat{\Gamma},\ell^2(\Gamma),{\mu}),$  the latter is equal to
$$(\pi(\gamma)\,\widetilde{T}^{-1}F_j,\widetilde{T}^{-1}F_h)
=(\widetilde{\pi}(\gamma)\,\widetilde{T}^{-1}F_j,\widetilde{T}^{-1}F_h)
=\int_{\widehat{\Gamma}}{(\gamma,\chi)\;
d\widetilde{\mu}_{\widetilde{T}^{-1}F_j,\widetilde{T}^{-1}F_h}(\chi)},
$$
see \eqref{value}. The measure
$\widetilde{\mu}_{\widetilde{T}^{-1}F_j,\widetilde{T}^{-1}F_h}$
is defined in any Borel set $E\subset\widehat{\Gamma},$ in terms of the 
projection valued
 measure $\widetilde{\Pi}$ associated with $\widetilde{\mu}$, see \eqref{meas}, by
\begin{eqnarray*}
\widetilde{\mu}_{\widetilde{T}^{-1}F_j,\widetilde{T}^{-1}F_h}(E)&=&
(\widetilde{\Pi}(E)\,\widetilde{T}^{-1}F_j,\widetilde{T}^{-1}F_h)\\
&=&(\widetilde{T}^{-1}(I_E\, F_j),\widetilde{T}^{-1}(F_h))\\
&=&(I_E\, F_j,F_h),
\end{eqnarray*}
the last two lines justified by the commuting properties of $\widetilde{T}^{-1}$
between $\widetilde{\Pi}$ and the canonical projection valued measure, and the unitariness of
the operator $\widetilde{T}^{-1}.$

But, for our choice of $F_1,\dots,F_n\in M_{\widetilde{J}},$
\begin{eqnarray*}
(I_E\, F_j,F_h)&=&\int_{E}{(F_j(\chi),F_h(\chi))\; d\widetilde{\mu}(\chi)}=
\int_{E\cap\widetilde{\tau}_n}{(F_j(\chi),F_h(\chi))\; d\widetilde{\mu}(\chi)}\\
&=&\delta_{j,h} \,\widetilde{\mu}(E\cap\tau_n).
\end{eqnarray*}

Hence the measure $\widetilde{\mu}_{\widetilde{T}^{-1}F_j,\widetilde{T}^{-1}F_h}$
is  identically zero whenever $j\neq h$, while for $j=h$ is nothing else that
$\widetilde{\mu}$ restricted to $\widetilde{\tau}_n$.

It follows that \eqref{point} is  identically zero whenever $j\neq h$,
while for $j=h,$ if we call $w$ the nonnegative measurable function on
$\widehat{\Gamma},$ provided by the Radon-Nikodym Theorem, such that $\widetilde{\mu}(E)=\int_E{w(\chi)\, d\mu(\chi)},$ for any Borel set $E$,
\eqref{point} is equal to
$$\int_{\widehat{\Gamma}}{(\gamma,\chi)\,
\|T\circ\widetilde{T}^{-1}(F_j)(\chi)\|^2\; d\mu(\chi)}=
\int_{\widehat{\Gamma}}{(\gamma,\chi)\;
I_{\widetilde{\tau}_n}(\chi)\;w(\chi)\;d{\mu}(\chi)}.$$

By Lemma \ref{comp}, we get  for  $\mu$ (and hence $\widetilde{\mu}$)
almost all $\chi\in\widetilde{\tau}_n,$
$$(T\circ\widetilde{T}^{-1}(F_j)(\chi),T\circ\widetilde{T}^{-1}(F_h)
(\chi))=\delta_{j,h}\,w(\chi).$$

The set where $w(\chi)=0$ has   zero measure with respect to
$\widetilde{\mu}$, so we get for $\widetilde{\mu}$
almost all $\chi\in\widetilde{\tau}_n,$
$0\not\equiv T\circ\widetilde{T}^{-1}(F_j)(\chi)\in J_V(\chi),$ and
$\dim J_V(\chi)\geq n.$

The same argument shows that in the set where $\dim \widetilde{J}(\chi)=+\infty$ then
also $\dim {J_V}(\chi)=+\infty,$ (see \cite[Theorem 2, p.8]{HLN}), hence  we get, $\widetilde{\mu}$
a.e.,   $\dim J_V(\chi)\geq \dim \widetilde{J}(\chi), $ as desired.
\end{proof}

\section{Main results}\label{main}
We first briefly discuss the linear independence of  translates
in $\ldg$
$\{T_{\gamma}f,\,\gamma\in\Gamma\}$, compared to linear independence of
$\{\pi{(\gamma)}\psi,\,\gamma\in\Gamma\}$, $\psi\in {\mathcal H}$ (recall that $T_{\gamma} f=f(\cdot-\gamma)$).

In \cite[Corollary 4.3.14]{Ku} Kutyniok has proved, among other things, that,
for any countable set $\Gamma\subset G$,
 the set of (left) translates
$\{T_{\gamma} f,\, \gamma\in \Gamma\}$
is linearly independent
for any $ 0\neq f\in L^2(G),$ if and only if, for any finite subset $\Lambda\subset
\Gamma,$ and
for any $(c_{\gamma})_{\gamma\in \Lambda}\subset \C,$   $(c_{\gamma})_{\gamma\in \Lambda}\neq 0,$ we have
$$\sum_{\gamma\in \Lambda}{c_{\gamma}\, ({\gamma},\chi)}\neq 0, \quad a.e.\; \chi\in\widehat{G},$$
where a.e. means with respect the Haar measure on $\widehat{G}$.

It turns out that the above equivalence  still holds if the  countable set $\Gamma\subset
G$ and the left translation are replaced, respectively,
by a countable, closed subgroup and a unitary representation with
corresponding measure $\mu$ absolutely continuous
as shown in the following lemma.

\begin{lem}\label{kuty}
Let $G$ be a locally compact abelian group. Let $\Gamma\subset G$ be a countable
closed subgroup. Let $\lambda$ be the Haar measure on $\widehat{\Gamma}$.

Then the following conditions are equivalent.
\begin{itemize}
\item[(i)] The set of (left) translates
$\{T_{\gamma} f,\, \gamma\in \Gamma\}$  is linearly independent
for any $ 0\neq f\in L^2(G);$
\item[(ii)] For any finite set $\Lambda\subset \Gamma$, and for any
$(c_{\gamma})_{\gamma\in \Lambda}\subset \C,$   $(c_{\gamma})_{\gamma\in \Lambda}\neq 0,$ we have
$$\sum_{\gamma\in \Lambda}{c_{\gamma}\, ({\gamma},\chi)}\neq 0, \quad \lambda-
a.e.\; \chi\in\widehat{\Gamma};$$
\item[(iii)] For any separable Hilbert space ${\mathcal K}$ and unitary representation
$\pi:\Gamma\rightarrow\mathcal{U}({\mathcal K}),$
with
corresponding measure $\mu\ll\lambda,$ and for any $0\neq f\in {\mathcal K},$  the set
$\{\pi(\gamma) f,\, \gamma\in \Gamma\}$ is linearly independent;
\item[(iv)]For any unitary representation
$\pi:\Gamma\rightarrow\mathcal{U}(L^2(G)),$
with
corresponding measure $\mu\ll\lambda,$ and for any $0\neq f\in L^2(G),$  the set
$\{\pi(\gamma) f,\, \gamma\in \Gamma\}$ is linearly independent.
\end{itemize}
\end{lem}
\begin{proof} (i) implies (ii) is part of the statement of
 Corollary 4.3.14 in \cite{Ku} together with the observation that the hypotheses on $\Gamma$ guarantee
 that any character of $\Gamma$ extends to a character of $G$, \cite{Ru}. Obviously (iii) implies (iv) which implies (i), so we need to show
 (ii) implies (iii). To this purpose we recall the unitary operator $T$ defined in
 \eqref{ti} and associated with $\pi$. Let $0\neq f\in {\mathcal K},$ and $F\subset\Gamma$ be
 a finite subset. If $(c_{\gamma})_{\gamma\in F}\subset \C,$
 $(c_{\gamma})_{\gamma\in F}\neq 0,$ is such that
 $\sum_{\gamma\in F}{c_{\gamma}\pi(\gamma)f}\equiv 0,$ we get that
 $0\equiv
 \sum_{\gamma\in F}{c_{\gamma}\, T(\pi(\gamma)f)}$
 implies
$$
 \left(\sum_{\gamma\in F}{c_{\gamma}\,(\gamma,\chi)}\right)
 [ T f]_i(\chi) =\sum_{\gamma\in F}{c_{\gamma}\,[ T(\pi(\gamma)f)]_i}(\chi)= 0,\quad
 \mu-\text{a.e.}\, \chi\in\widehat{\Gamma},
$$
 for any index $i$ arising in the definition of $T$.
 By (ii) and absolute continuity of the measure, the sum $\left(\sum_{\gamma\in F}{c_{\gamma}\,(\gamma,\chi)}\right)$ is
 $\mu$-a.e. different from zero, yielding $[ T f]_i(\chi)=0$ for all $i$ and
 $\mu$-almost all
 $\chi\in\widehat{\Gamma}.$ This implies $f\equiv 0,$ since $T$ is unitary,
 against the assumption $f\neq 0$.
\end{proof}

We now return to questions about linear independence of the  affine system $Y$
and the role played by the endomorphism $\alpha$ defined in Section \ref{Hypo}.

A little technical lemma anticipates one of the main results, which explores the
behavior of the space $V_0$ of negative translates, as in \cite[Theorem 3.4]{BoSp}.
\begin{lem}\label{lemma1}
Under the hypotheses \eqref{Ba} on $\alpha$ and $\alpha(\Gamma),$ there exists
a finite set $\nu_1,\dots,\nu_N\in\Gamma$ with the following property:
 for any
$\gamma\in\Gamma,$ and for any $j\in \N,$
there exists $\eta\in\Gamma$  such that $\gamma=\alpha^j(\eta)\nu_i$ for some $i=1,\dots,N.$
\end{lem}
\begin{proof}
By a recursive argument it is sufficient to consider $j=1.$

Assume $$\left|\Gamma/{\alpha(\Gamma)}\right|=N,$$ and let $\nu_1,\nu_2,\dots, \nu_N\in\Gamma$
be a complete set of coset representatives. If $\gamma\in\Gamma,$ let $i=1,\dots,N$ such that
$[\gamma]=[\nu_i].$ Then $\gamma {\nu_i}^{-1}\in\alpha(\Gamma)$ and so $\gamma=\alpha(\eta)\nu_i$
for some $\eta\in\Gamma.$
\end{proof}

As in the previous section we denote by $T$, $\mu$ and $m$ the usual objects
linked to the representation $\pi$.

The following remark is crucial in the proof of what follows.

We observe that the compatibility condition \eqref{Ba1} implies that, for any
$M\in\N$, the representation $\pi\circ\alpha^M:\Gamma\rightarrow{\mathcal U}({\mathcal H}),$ is
 unitarily equivalent to $\pi$, being $\delta^M:{\mathcal H}\rightarrow{\mathcal H}$
 the intertwining operator. It follows that the corresponding measures $\mu$ and,
say $\mu_M$, are
equivalent and the multiplicity functions agree up to a set of measure $0$ (with respect to either
measure). In particular,  $\mu_M$ is absolutely continuous.

 The compatibility condition \eqref{Ba} is satisfied too, since $\pi$ does,
 $$\delta^{-1}\pi(\alpha^M(\gamma))\delta=
 \pi(\alpha(\alpha^M(\gamma)))=\pi\circ\alpha^M(\alpha(\gamma)).$$

We can conclude that all results of this paper obtained so far hold for the representation
$\pi\circ \alpha^M$ as well.

\begin{thm}\label{teorema3.4}
Assume hypothesis {\bf{(A)}}.

If the system $Y=\{\delta^j\, \pi(\gamma)\psi, \;{j\in\Z},\, \gamma\in\Gamma\}$
is linearly dependent, then $V_0=\overline{\text{span}}\{\delta^j\, \pi(\gamma)\psi, \; j<0,\,
\gamma\in\Gamma\}$ is $\pi\circ\alpha^M$-invariant for some $M\geq 0.$

Moreover, if we
consider the  unitary representation
$$\pi_M:=\pi\circ\alpha^M:\Gamma\rightarrow{\mathcal U}({\mathcal H}),$$
and the subrepresentation $\widetilde{\pi_M}$ on $V_0$, with
corresponding measure $\widetilde{\mu_M}$ and multiplicity function $\widetilde{{m}_M},$ we have
$$
\widetilde{{m}_M}(\chi)<+\infty,\quad \widetilde{\mu_M}-a.e.\, \chi\in\widehat{\Gamma}.
$$
\end{thm}
\begin{proof}
The first part follows the proof of Theorem 3.4 in \cite{BoSp}.

If the system $Y$ is linearly dependent, there exists a finite set $F\subset\Gamma$
and a finite non zero sequence
$c_{j,\gamma}\in\C,$ $j\in\Z,$ $\gamma\in F,$ such that
\begin{equation}\label{sommazero}
0=\sum_{j\in\mathbb{Z}}\sum_{\gamma\in F}{c_{j,\gamma}\,\delta^j\, \pi(\gamma)\psi}.
\end{equation}
After several applications of $\delta,$ we can assume that the smallest $j$ in the
sum \eqref{sommazero} is $0.$ Call the largest $M$. Hence
$$0=\sum_{j=0}^M\sum_{\gamma\in F}{c_{j,\gamma}\,\delta^j\, \pi(\gamma)\psi}.$$
We define
\begin{equation}\label{effe}
f:=\sum_{\gamma\in F}{c_{0,\gamma}\, \pi(\gamma)\psi}=
-\sum_{j=1}^M\sum_{\gamma\in F}{c_{j,\gamma}\,\delta^j\, \pi(\gamma)\psi}.
\end{equation}
For any $h,k\in\Z,$ we consider the subspaces in ${\mathcal H}$
$$
V_{h,k}=\overline{\text{span}}\{\delta^j\, \pi(\gamma)\psi, \;h\leq j\leq k,\,
\gamma\in\Gamma\},$$ and we note first that $f\in V_{1,M},$ and
$\delta(V_{h,k})=V_{h+1,k+1};$
secondly, each $V_{h,k}$ is $\pi$-invariant whenever $h\geq 0.$ Indeed, by
\eqref{Ba1},
$$\pi(\eta)\delta^j\, \pi(\gamma)\psi=
\delta^j\pi(\alpha^j(\eta))\, \pi(\gamma)\psi=
\delta^j\pi(\alpha^j(\eta)\gamma)\psi\in V_{h,k}.$$
By \eqref{effe} and \eqref{inter}, we get, for all $i$ and $\mu$-a.e. $\chi\in\sigma_i,$
$$
[Tf]_i(\chi)=\sum_{\gamma\in\Gamma}{c_{0,\gamma}\,[T(\pi(\gamma)\psi)]_i(\chi)}
=\left(\sum_{\gamma\in\Gamma}{c_{0,\gamma}\,(\gamma,\chi)}\right)[T(\psi)]_i(\chi).
$$
By (ii) of Lemma \ref{kuty}, the hypothesis ({\bf{A}}) on linear independence of translates implies that
$$\sum_{\gamma\in\Gamma}{c_{0,\gamma}\,(\gamma,\chi)}\neq 0,\quad\lambda\text{-a.e.}
\, \chi\in\widehat\Gamma,$$
($\lambda$ is the Haar measure) hence, by absolute continuity, $\mu$-a.e.
$$[T(\psi)]_i(\chi)=\frac{1}{\sum_{\gamma\in\Gamma}{c_{0,\gamma}\,(\gamma,\chi)}}
\,[Tf]_i(\chi),$$
and we obtain, by Corollary \ref{coro1}, that $\psi\in V_{1,M}.$

Therefore,
since $V_{1,M}$ is $\pi$-invariant, $\{\pi(\gamma)\psi,\,\gamma\in\Gamma\}\subset V_{1,M},$
yielding $V_{0,M}\subset V_{1,M}$, and so $V_{1,M}=V_{0,M}.$ By several application
of $\delta^k,$ we get also $V_{k+1,M+k}=V_{k,M+k}.$ The argument goes on as in
 the proof of \cite{BoSp}, we include it for completeness.
 By induction it is proved that
 $$ V_{r,M}=V_{1,M},\quad \text{for all}\, r\leq 0.$$

 Indeed, by above the statement is true for $r=0.$ Suppose it is true for  $r+1\leq 0,$  and
 consider $r\leq 0$, then obviously $V_{r+1,M+r}\subset V_{r+1,M}$ and we have
 $$V_{r,M+r}=V_{r+1,M+r}\subset V_{r+1,M}=V_{1,M},$$
 the latter equality being the induction hypothesis.

 So the inclusion
 \begin{eqnarray*}
 \{\delta^j\pi(\gamma)\psi,\, r\leq j\leq M,\,\gamma\in\Gamma\}&\subset&
 \{\delta^j\pi(\gamma)\psi,\,r\leq j\leq M+r,\,\gamma\in\Gamma\}\cup\\
 & & \{\delta^j\pi(\gamma)\psi,\,r+1\leq j\leq M,\,\gamma\in\Gamma\}
 \end{eqnarray*}
 implies, since $r\leq 0,$
 \begin{eqnarray*}
 V_{r,M}&=&\overline{\text{span}}
 \{\delta^j\pi(\gamma)\psi,\, r\leq j\leq M,\,\gamma\in\Gamma\}\\
 &\subset& V_{1,M}\cup V_{r+1,M}=V_{1,M}\cup V_{1,M}=V_{1,M}\subset V_{r,M},
 \end{eqnarray*}
 as needed.

 Hence we obtain
 $$V_{M+1}=\overline{\text{span}}\{\delta^j\pi(\gamma)\psi,\, j\leq M,\,
 \gamma\in\Gamma\}=\bigcup_{r\leq 0}V_{r,M}=V_{1,M},$$
 and $V_0=\delta^{-(M+1)}(V_{M+1})=V_{-M,-1}.$

 Now we prove that $V_0$ is $\pi\circ\alpha^M$-invariant. Indeed, for
 $-M\leq j\leq -1,$ and $\eta,\nu\in\Gamma,$ by  equality \eqref{Ba1} we have
 \begin{eqnarray*}
 \pi\circ\alpha^M(\eta)(\delta^j\pi(\nu)\psi)&=&
 \delta^{-M}\pi(\eta)\delta^{M}(\delta^j\pi(\nu)\psi)=
 \delta^{-M}\pi(\eta)\delta^{M+j}\pi(\nu)(\psi)\\
 &=&
 \delta^{-M}\delta^{M+j}\pi(\alpha^{M+j}(\eta))\pi(\nu)(\psi)=
  \delta^{j}\pi(\alpha^{M+j}(\eta)\nu)(\psi),
  \end{eqnarray*}
and the latter is again in $V_{-M,-1}=V_0.$ Furthermore,  by Lemma \ref{lemma1},
any element of $\Gamma$ is of the form $\alpha^{M+j}(\eta)\nu$, for $\nu$ varying
in a finite set and $\eta\in\Gamma,$ hence by the above equality we get also that
$$V_0=\overline{\text{span}}\{\pi\circ\alpha^M(\eta)(\delta^j\pi(\nu_i)\psi),
\,-M\leq j\leq -1,\, i=1,\dots,N,\,\eta\in\Gamma\}.$$

Finally let us consider the unitary representation
$\pi\circ\alpha^M:\Gamma\rightarrow\mathcal{U}({\mathcal H}),$
the associate unitary map ${T}_M$ as defined in \eqref{ti}
together with the measure $\mu_M.$ By Theorem \ref{hel},
if $J_{V_0,M}$ denotes a $\mu_M$ measurable range function corresponding to $V_0$
we have, $\mu_M$-a.e. $\chi\in\widehat\Gamma,$
$$J_{V_0,M}(\chi)=\overline{\text{span}}\{T_M(\delta^j\pi(\nu_i)\psi)(\chi),\,
-M\leq j\leq -1,\, i=1,\dots,N\}.$$
By Proposition \ref{dim}, we have $\dim J_{V_0,M}(\chi)<+\infty,$
and so, by Proposition \ref{dimult}, we get  $\widetilde{m_M}(\chi)<+\infty, $
$\widetilde{\mu_M}$-a.e., as required.
 \end{proof}

\begin{thm}\label{lemma3.5}
Let $\pi$ be a unitary representation of $\Gamma$ on ${\mathcal H}$ verifying \eqref{Ba},
and suppose that hypothesis {\bf{(A)}} holds true.
Assume $V\subset{\mathcal H}$ is $\pi$-invariant and $V=\delta V$.

Let $\mu$ and $\widetilde{\mu}$ denote the obvious measures
and assume that $\mu$ is absolutely continuous.
Let
 $\widetilde{m}$ be the multiplicity function associated with the subrepresentation of
$\pi$ on $V$. Then  we have, for $\widetilde{\mu}$-almost all
$\chi\in\widehat{\Gamma}$,
 $\widetilde{m}(\chi)=+\infty.$
\end{thm}
\begin{proof}

Since $V=\delta V,$ the multiplicity function, $m^{\sharp}$, associated with the subrepresentation of $\pi$ on $\delta V$ coincides with $\widetilde{m}.$
 Hence, by \eqref{formulaBa}   we get
 $\widetilde{\mu}$-almost
all $\chi\in\widehat{\Gamma},$
\begin{equation}\label{mtilde}
\widetilde{m}(\chi)=m^{\sharp}(\chi)=\sum_{\alpha^{\ast}(\xi)=\chi}{\widetilde{m}(\xi)}.
\end{equation}

Also, since  the measure $\widetilde{\mu}$ is absolutely continuous with respect to the Haar
measure  $\lambda$,
as discussed at the end of Section \ref{Hypo}, we can take  $\widetilde{\mu}$ to be
the restriction of $\lambda$  to the subset
$\widetilde{\sigma}_1$ provided by \eqref{sigma1}.

The proof follows now the same standard ergodic argument as in
\cite[Lemma 3.5]{BoSp}.

Let $E=\{\chi\in\widehat{\Gamma},\, \widetilde{m}(\chi)\geq 1\}$. Then
$E\subset\widetilde{\sigma}_1$ and
$
\widetilde{\mu}(E)=\lambda(E\cap\widetilde{\sigma}_1)
=\lambda(E).
$

By \eqref{mtilde} the inclusion
 $E\subset(\alpha^{\ast})^{-1}(E)$ holds $\widetilde{\mu}$-a.e.,
hence, by above, $\lambda$-a.e.

Since $\alpha^{\ast}$ is measure-preserving,in the sense that
$\lambda((\alpha^{\ast})^{-1}(E))=\lambda(E),$ we have,
modulo $\lambda$-null-sets,  $E=(\alpha^{\ast})^{-1}(E).$
Since $\alpha^{\ast}$ is ergodic with respect to Haar measure, we must have
either $\lambda(E)=0$ or $\lambda(E)=1$, and since $V\neq \{0\},$ it follows that
$\lambda(E)=1$.

So $\widetilde{m}(\chi)\geq 1$ for $\lambda$-a.e. $\chi\in\widehat{\Gamma}.$
From \eqref{mtilde} above and the fact that
all elements $\xi$ verifying  $\alpha^{\ast}(\xi)=\chi$ yield the same coset
in $\widehat{\Gamma}/\ker \alpha^{\ast},$
it follows first that $\widetilde{m}(\chi)\geq N>1$ for $\lambda$-a.e. $\chi\in\widehat{\Gamma},$ ($N=|\ker \alpha^{\ast}|$) and secondly
 $\widetilde{m}(\chi)=+\infty$ for $\lambda$-a.e. $\chi\in\widehat{\Gamma}.$
The fact that $\widetilde{\mu}\ll\mu,$
see Remark \ref{rap}, complete the proof.
\end{proof}
{\bf Proof of Theorem \ref{main1}}
By hypothesis $V_0$ is $\pi$-invariant.\\
If $V_0\neq V_1=\delta(V_0)$ then Theorem \ref{primali} yields the linear independence of $Y$.\\
If $V_0=V_1,$ then the restriction of $\delta$ to $V_0$, say $\delta_{V_0}$, is a
unitary map onto $V_0.$

If we assume that $Y$ is linear dependent,
by Theorem \ref{teorema3.4} there exists an $M\geq 0$ such that $V_0$ is
$\pi_M:=\pi\circ\alpha^M$-invariant, and the multiplicity function  $\widetilde{m_M}$
verifies $\widetilde{m_M}(\chi)<+\infty,$ for $\widetilde{\mu_M}$-almost all $\chi\in\widehat{\Gamma}$.

But the $\pi_M$-invariance implies that
the subrepresentations $\widetilde{\pi_M}$ and $\widetilde{\pi}$ on $V_0$ are equivalent,
$\delta_{V_0}^M$ being the intertwining operator, and so the corresponding measures
$\widetilde{\mu_M}$  and $\widetilde\mu$ are
equivalent. Furthermore the multiplicity functions agree up to a set of measure $0$ (with respect to either
measure), so
$\widetilde{m}(\chi)<+\infty,$ for $\widetilde\mu$-almost all $\chi\in\widehat{\Gamma}$,
leading to a contradiction of  Theorem \ref{lemma3.5}.
\qed

{\bf Proof of Theorem \ref{main2}}\\
It follows by Theorem \ref{parse} and Theorem \ref{main1}.
\qed
\vspace{.5cm}\\
We conclude the paper with a non trivial example. 
\begin{ex}

 Fix an integer $a>1$ and real numbers $b, c>0$. Let $\psi\in L^2(\mathbb{R})$ such
that $\text{supp}(\widehat{\psi})\subset[c,c+b^{-1})$ and
$\displaystyle{\sum_{n\in\mathbb{Z}}|\widehat{\psi}(a^n\xi)|^2}=b$ for almost every $\xi\geq0.$
The system $a^{n/2}\psi (a^nx-bk),$ $k,n\in\Z$, is a Parseval frame for the Hilbert space
$H^2_+(\R)=\{f\in\ld,\;\text{supp}\widehat{f}\subset[0,+\infty)\},$
see  Heil's book \cite[ex.12.3]{Hei}.

So we can take $G=\R$, $\Gamma=b\Z$ acting by translations on $\mathcal{H}=H^2_+(\R)$,
$\delta=D_a,$ and $\alpha(bk)=abk.$ All conditions required by Theorem \ref{main2} are satisfied.  Hence we can conclude that
the system  $\{a^{n/2}\psi (a^n\cdot-bk)=\delta^n\pi(bk)\psi,\;n,k\in\Z\}$ is linearly
independent.

\end{ex}


\begin{thebibliography}{30}

\bibitem{Ba4}
L. W. Baggett, V. Furst,  K. D. Merrill, J. A. Packer,
\textit{Classification of generalized  multiresolution analyses},
J. Funct. Anal. \textbf{258} (2010), 4210--4228.

\bibitem{Ba}
L. W. Baggett, N. S. Larsen, K. D. Merrill, J. A. Packer, I. Raeburn,
\textit{Generalized multiresolution analyses with given multiplicity
functions}, J. Fourier Anal. Appl. \textbf{15} (2009), 616--633.

\bibitem{BHP}
D. Barbieri, H. Hern\'{a}ndez, V. Paternostro,
\textit{The Zak transform and the structure of spaces invariant by the action of an LCA group
}, J. Funct. Anal.
\textbf{269} (2015), 1327–-1358.

\bibitem{Bo}
M. Bownik,
\textit{The structure of shift invariant subspaces of $L^2(\mathbb{R}^n)$},
J. Funct. Anal.
\textbf{176} (2000), 1–-28.
%
\bibitem{BoRo}
M. Bownik,  K. A. Ross,
\textit{The structure of translation-invariant spaces on locally compact abelian groups}, J. Fourier Anal. Appl. \textbf{21} (2015), 849--884.

\bibitem{BoRz}
M. Bownik, Z. Rzeszotnik,
\textit{The spectral function of shift-invariant spaces},
Michigan Math. J. \textbf{51} (2003), 387–-414.

%
\bibitem{BoSp}
M. Bownik,  D. Speegle,
\textit{Linear independence of Parseval wavelets},
Illinois J. Math. \textbf{54}  (2010), 771--785.


\bibitem{BoWe}
M. Bownik,  E. Weber,
\textit{Affine frames, GMRA's, and the canonical dual},
Studia Math. \textbf{159}  (2003), 453--479.
%
%
\bibitem{CaPa}
C. Cabrelli, V. Paternostro,
\textit{Shift-invariant spaces on LCA groups},
J. Funct. Anal. \textbf{258}  (2010), 2034--2059.

\bibitem{BoDeRo}
C. de Boor, R. A. DeVore, A. Ron,
\textit{The structure of finitely generated shift-invariant spaces in $L^2(\R^d)$}, J. Funct. Anal.,
\textbf{119}  (1994), 37--78.

\bibitem{EdRo79}
G.  A. Edgar, J. M. Rosenblatt,
\textit{Difference equations over locally compact abelian groups}
Trans. Amer. Math. Soc. \textbf{253} (1979), 273--289.


\bibitem{F}
G. B. Folland,
\textit{A Course in Abstract Harmonic Analysis},
CRC Press, 1995.
%
\bibitem{H64}
H. Helson,
\textit{Lectures on Invariant Subspaces},
Academic Press, 1964.
\bibitem{HLN}
H. Helson,
\textit{The Spectral Theorem},
Lecture Notes in Math. 1227, Springer-Verlag, 1986.

\bibitem{HeRo}
E. Hewitt, K. A. Ross,
\textit{Abstract harmonic analysis}, vols. 1 and 2, Springer-Verlag, New
York, 1963.
%
\bibitem{Hei}
C. Heil,
\textit{A basis theory primer}, Birkh\"auser/Springer, New
York, 2011.

\bibitem{I}
J. W. Iverson,
\textit{Subspaces of $L^2(G)$ invariant under translations by an abelian subgroup},
J. Funct. Anal \textbf{269}  (2015), 865–-913.

\bibitem{KaRa}
R. A. Kamyabi Gol, R. Raisi Tousi,
\textit{The structure of shift invariant spaces on a locally compact abelian group},
J. Math. Anal. Appl.,
\textbf{340}  (2008), 219--225.

\bibitem{KaRa10}
R. A. Kamyabi Gol, R. Raisi Tousi,
\textit{A range function approach to shift-invariant spaces on locally compact abelian groups},
Int. J. Wavelets Multiresolut. Inf. Process.,
\textbf{8}  (2010), 49--59.

\bibitem{Ku}
G. Kutyniok,\textit{Time-frequency analysis on locally compact groups}, Ph.D.
thesis, University of Paderborn, Germany, 2000.


\bibitem{Ma}
G. W. Mackey,
\textit{Induced representation of locally compact groups. I},
Ann. Math. (2), \textbf{55} (1952), 101--139.

\bibitem{Ro08}
J. Rosenblatt,
\textit{Linear independence of translations},
Int. J. Pure Appl. Math.,
\textbf{45}  (2008), 463--473.

\bibitem{Ru}
W. Rudin,
\textit{Fourier Analysis on Groups},
John Wiley, 1962.

\bibitem{We}
E. Weber,
\textit{The action of translations on wavelet subspaces},
 Ph.D. thesis, University of Colorado, Boulder, 1999.


\end{thebibliography}
\end{document}